\documentclass[12pt,oneside,letterpaper,reqno]{amsart}
\usepackage{amssymb}
\usepackage[left=1in,top=1in,right=1in,bottom=1in]{geometry}
\newcommand{\dup}[2]{\langle#1,#2 \rangle}

\newcommand{\ip}[2]{(#1,#2)}
\newcommand{\Co}{\mathrm{Co}} 
\newcommand{\cdual}{{\ast}}
\newcommand{\supp}{\mathrm{supp}}
\renewcommand{\theenumi}{\alph{enumi}} 
\newtheorem{theorem}{Theorem}[section]
\newtheorem{lemma}[theorem]{Lemma}
\newtheorem{assumption}[theorem]{Assumption}
\newtheorem{corollary}[theorem]{Corollary}

 \theoremstyle{definition}
\newtheorem{definition}[theorem]{Definition}

\newtheorem{remark}[theorem]{Remark}

\newcommand{\Mid}{\,\Big|\,}

\begin{document}

\title{Sampling in reproducing kernel {B}anach spaces on Lie groups} 
\subjclass[2000]{Primary
  43A15,46E15,94A12; Secondary 46E22} 
\keywords{Sampling, Reproducing kernel Banach spaces,
  Lie Groups, Coorbits} 

\author{Jens Gerlach Christensen} 
\address{2307 Mathematics Building, Department of mathematics, 
  University of Maryland, College Park} 
\email{jens@math.umd.edu}
\urladdr{http://www.math.umd.edu/\textasciitilde jens}

\thanks{The  author gratefully acknowledges support from the
  NSF grant DMS-0801010 and ONR grant NAVY.N0001409103}

\begin{abstract}
We present sampling theorems for reproducing kernel Banach
spaces on Lie groups. Recent approaches to this problem
rely on integrability of the kernel and its local oscillations. 
In this paper we replace the integrability conditions
by requirements on the derivatives of the reproducing kernel.
The results are then used to obtain frames and atomic decompositions
for Banach spaces of distributions stemming from a cyclic representation,
and it is shown that this is particularly easy, when the cyclic vector is
a G{\aa}rding vector for a square integrable representation.
\end{abstract}

\maketitle
\begin{center} {\today}
\end{center}

\section{Introduction}
The classical sampling theorem for band-limited functions 
states that a function can be reproduced from its samples at 
equidistant points. At the core of this statement lies the fact
that a bounded interval has an orthonormal basis of complex exponentials.
Extensions of this theorem for irregular sampling points
have been found using the smoothness of the functions involved
\cite{Grochenig1992,Feichtinger1992,Grochenig1993}. The irregularity 
and density
of the sampling points is connected to the theory of frames 
\cite{Duffin1952,Benedetto1990}: 
A sequence of vectors
$\phi_i$ in a Hilbert space $H$ is called a frame, if there are constants
$0<A\leq B<\infty$ such that
\begin{equation*}
  A \| f\|^2 \leq \sum_i |\ip{f}{\phi_i}|^2 \leq B \| f\|^2
\end{equation*}
for all $f\in H$.
A vector $f$ can be reconstructed by inversion of the frame operator
\begin{equation*}
  Sf = \sum_i \ip{f}{\phi_i}\phi_i
\end{equation*}
A Banach (or Hilbert) space of 
functions on a set $D$ for which
point evaluation is continuous is called a 
reproducing kernel Banach (or Hilbert) space.
Sampling at points $x_i$ provide a frame on a reproducing kernel Hilbert space
$H$ if for all $f\in H$
\begin{equation}
  \label{eq:4}
  A \| f\|^2 \leq \sum_i |c_i f(x_i)|^2 \leq B \| f\|^2
\end{equation}
where $c_i$ are constants.
If this frame inequality is satisfied we can reconstruct
$f$ from its samples $f(x_i)$. For reproducing kernel Banach
spaces the existence of a 
reconstruction operator is not evident from a frame type
inequality. However in \cite{Han2009} it was proven that reconstruction
is possible if for $1\leq p <\infty$ there are 
constants $0<A\leq B <\infty$ such that
\begin{equation}
  A \| f\|^p \leq \sum_i |c_i f(x_i)|^p \leq B \| f\|^p
\end{equation}
for all $f\in B$.
For other types of reproducing kernel Banach spaces more care
has to be taken and a lot more machinery is needed.
The article \cite{Nashed2010} is concerned with reconstruction in
reproducing kernel subspaces of $L^p(\mathbb{R}^n)$
and \cite{Feichtinger1989a,Grochenig1991} 
deals with Banach spaces defined via representations 
of locally compact groups. 
Common for these approaches is that a reproducing kernel is given by 
an integral over a locally compact group
\begin{equation*}
  f(x) = \int_G f(y)K(x,y)\,dy
\end{equation*}
This kernel is assumed to be integrable, i.e. for every $x$
\begin{equation*}
  \int_G |K(x,y)|\,dy < \infty 
\end{equation*}
It is also assumed that for a compact set $U$,
the local oscillations
\begin{equation*}
  M_U K(x,y) = \sup_{u,v\in U} |K(xu,yv)-K(x,y)|
\end{equation*}
satisfy
\begin{equation*}
  \int_G |M_U K(x,y)|\,dy < \infty 
\end{equation*}
These assumptions are not satisfied for band-limited functions, since
the reproducing kernel is the non-integrable $\mathrm{sinc}$-function.
Other non-integrable kernels are known (see for example
the sections about Bergman spaces in 
\cite{Christensen2009,Christensen2010})
and this calls for a sampling theory without
integrability conditions. 
The main idea in this article is to estimate local oscillations
via derivatives, and therefore we restrict our attention to reproducing 
kernel Banach spaces on Lie groups.

Reproducing kernel Banach spaces show up naturally in connection
with square integrable representations, which 
was first noted in the construction of coorbit spaces (see 
\cite{Feichtinger1988,Feichtinger1989a}). 
In \cite{Christensen2009,Christensen2010} this work was generalized
and coorbit spaces were defined as Banach spaces of 
distributions stemming from cyclic representations.
As an application of our sampling theorems we obtain frames and 
atomic decompositions for coorbit spaces arising from cyclic 
(and not necessarily integrable) representations of Lie groups.

\section{Examples with reproducing kernel Hilbert spaces}

In this section we will cover sampling theorems for two cases of
reproducing kernel Hilber spaces on groups. The two groups are
$\mathbb{R}$ and the $(ax+b)$-group.

\subsection{Sampling of band-limited functions}

The Fourier transform is the extension to $L^2(\mathbb{R}^n)$ of
the operator $\mathcal{F}$
\begin{equation*}
  \mathcal{F}f(w) = (2\pi)^{-n/2} \int f(x)e^{-iw\cdot x}\,dx
\end{equation*}
defined for $f\in L^1(\mathbb{R})\cap L^2(\mathbb{R})$.
We will often denote the Fourier transform $\mathcal{F}f$ by $\widehat{f}$.
A function in $L^2(\mathbb{R})$ is called $\Omega$-band-limited if 
$\mathrm{supp}(\widehat{f})\subseteq [-\Omega,\Omega]$.
The space $L^2_\Omega$ of $\Omega$-band-limited functions is a reproducing
kernel Hilbert space and satisfies
\begin{equation*}
  f(x) = \int f(y)\mathrm{sinc}(x-y) \,dy
\end{equation*}
where $$\mathrm{sinc}(x) = \frac{\sin x}{x}$$
Therefore we need only provide a frame inequality like 
\eqref{eq:4} in order to reconstruct $\Omega$-band-limited functions.
In \cite{Grochenig1993} the following irregular sampling theorem for
band-limited functions was used to provide sampling theorems
for the wavelet and short time Fourier transforms.
\begin{theorem}
  \label{thm:10}
  Suppose that $f\in L^2(\mathbb{R})$ and 
  $\mathrm{supp}(\widehat{f}) \subseteq [-\Omega,\Omega]$.
  If $\{x_k \}_{k\in\mathbb{Z}}$ is any increasing sequence such that
  the maximal gap length
  $\delta$ satisfies
  \begin{equation*}
    \delta := \sup_{k\in\mathbb{Z}} (x_{k+1}-x_{k}) < \frac{\pi}{\Omega}
  \end{equation*}
  then
  \begin{equation*}
    (1-\delta\Omega/\pi)^2 \| f\|_2^2
    \leq 
    \sum_k \frac{x_{k+1}-x_{k-1}}{2}|f(x_k)|^2
    \leq
    (1+\delta\Omega/\pi)^2 \| f\|_2^2
  \end{equation*}
\end{theorem}
To prove this it is first shown that
for disjoint intervals $I_k \subseteq (x_k-\delta/2,x_k+\delta/2)$
with $\cup_k I_k = \mathbb{R}$
we have
\begin{equation} \label{eq:1}
  \left\| f - \sum_k f(x_k)1_{I_k}\right\|_{L^2} 
  \leq \frac{\delta}{\pi} \| f' \|_{L^2}
\end{equation}
This inequality follows from an application of Wirtinger's inequality.
Then Bernstein's inequality
$\| f'\|_{L^2} \leq \Omega \| f\|_{L^2}$  is utilized to obtain
the frame inequality of the theorem above.

We now give an alternative approach to 
inequalities resembling \eqref{eq:1}. 
Note that this has already been presented
as Lemma~4 in \cite{Grochenig1992}, however we include the 
calculations here to demnonstrate how they can be generalized.
This is more straight forward
than Wirtinger's inequality and uses the smoothness of band-limited 
functions and the fundamental theorem of calculus. 
Since many reproducing kernel spaces consist of 
differentiable functions this approach will carry over to such spaces. 
Define the local oscillation of a band-limited function
$f$ as
\begin{equation*}
  M^\delta f(x) = \sup_{|u|\leq \delta} |f(x+u)-f(x)|
\end{equation*}
Then an application of H\"older's inequality shows that
\begin{align*}
  M^\delta f(x) 
  &= \sup_{|u|\leq \delta} |f(x+u)-f(x)| \\
  &= \sup_{|u|\leq \delta} \left|\int_{0}^{u} f'(x+t)\,dt \right| \\
  &\leq \sup_{|u|\leq \delta}
  \left( 
    \int_{0}^{|u|} 1 \,dt  \right)^{1/2} 
  \left(
    \int_{0}^{u} |f'(x+t)|^2 \,dt
  \right)^{1/2}  \\
  &\leq 
  \delta  
  \left(   
  \int_{-\delta}^{\delta} |f'(x+t)|^2 \,dt
  \right)^{1/2}
\end{align*}
Applying Minkowski's inequality then gives the following oscillation estimate
\begin{equation}
  \label{eq:2}
  \| M^\delta f\|_{L^2} \leq \sqrt{2}\delta \| f'\|_{L^2}
\end{equation}
From this follows
\begin{equation*}
  \left\| f - \sum_k f(x_k)1_{I_k}\right\|_{L^2} 
  \leq \| M^\delta f\|_{L^2}
  \leq \sqrt{2}\delta \| f' \|_{L^2}
\end{equation*}
and we can again derive a frame inequality by use of Bernstein's inequality.
Note that this estimate is not as sharp as
\eqref{eq:1}, however it has the advantage that
it can be generalized to other groups than $\mathbb{R}$.

In this paper we will derive oscillation estimates similar to 
\eqref{eq:2} for (non-commutative) Lie groups in order to 
obtain sampling theorems. 
In the next subsection we 
work through the details for the non-commutative 
$(ax+b)$-group and show how this provides sampling theorems
for the wavelet transform.

\subsection{Sampling of the wavelet transform}

In this section we present the ideas behind sampling for
reproducing kernel Hilbert space
related to the non-commutative $(ax+b)$-group.
The approach will be
generalized in section~\ref{sec:srkbs}.

Let $G$ be the $(ax+b)$-group which can be realized as a matrix 
group
\begin{equation*}
  G = \left\{ 
  (a,b) = \begin{pmatrix}
    a & b \\ 0 & 1
  \end{pmatrix}
 \Mid a> 0,b\in\mathbb{R}   \right\}
\end{equation*}
The left Haar measure on $G$ is defined by
\begin{equation*}
  C_c(G)\ni f\to \int_{0}^\infty \int_{\mathbb{R}}
  f(a,b) \,\frac{da\,db}{a^2}
\end{equation*}
and we denote by $L^2(G)$ 
the space of square integrable functions with respect to this measure.
For a function $g$ let $g^\vee$ be the function
\begin{equation*}
  g^\vee (x) = g(x^{-1})
\end{equation*}
Convolution of two functions $f,g^\vee \in L^2(G)$ is given by
\begin{equation*}
  f*g(a,b) 
  =  \int_{0}^\infty \int_{\mathbb{R}} 
  f(a_1,b_1) g((a_1,b_1)^{-1}(a,b)) \,\frac{da\,db}{a^2}
\end{equation*}
Assume that $\phi\in L^2(G)$ is a non-zero function for which
$\phi^\vee \in L^2(G)$ and the mapping
\begin{equation*}
  L^2(G) \ni f\mapsto f*\phi \in L^2(G)
\end{equation*}
is continuous. Further assume that $\phi*\phi = \phi$,
then the space $H_\phi =L^2(G)*\phi$ is a reproducing kernel
Hilbert space and the reproducing kernel is given by convolution
with $\phi$. 
In order to obtain oscillation estimates we need some notation 
concerning differentiation.
The Lie algebra of $G$ is
\begin{equation*}
  \mathfrak{g} = \left\{ 
    \begin{pmatrix}
      s & t \\ 0 & 0
    \end{pmatrix} \Mid s,t\in \mathbb{R}
  \right\}
\end{equation*}
and the exponential function is the usual matrix exponential function
\begin{equation*}
  e^A = \sum_{k=0}^\infty A^k/k!
\end{equation*}
For $X\in \mathfrak{g}$ define the differential operator
\begin{equation*}
  Xf(x) = \frac{d}{dt}\Big|_{t=0} f(xe^{tX})
\end{equation*}
Denote by $X_1,X_2$ the basis for the Lie algebra $\mathfrak{g}$ of
$G$ for which 
$$
e^{t X_1} = 
\begin{pmatrix}
  e^t & 0 \\ 0 & 1
\end{pmatrix}
\qquad\text{and}\qquad
e^{tX_2} = 
\begin{pmatrix}
  1 & t \\ 0 & 1
\end{pmatrix}
$$ 
For $\alpha$ a
multi-index of length $k$ with entries $1$ or $2$ 
we define the differential operators $R^\alpha$
\begin{equation*}
  R^\alpha f = X_{\alpha(k)}X_{\alpha(k-1)}\cdots X_{\alpha(1)} f
\end{equation*}

When $\epsilon$ is a positive number we define the neighbourhood 
$U_\epsilon$ of the identity $e$ by
\begin{equation*}
  U_\epsilon = \{ e^{t_1X_1}e^{t_2X_2} \mid -\epsilon < t_1,t_2 < \epsilon \}
\end{equation*}
Choose points $x_i\in G$ such that
$G\subseteq \cup_i x_iU_\epsilon$. Let 
$U_i\subseteq x_iU_\epsilon$ be disjoint sets such that $G=\cup_i U_i$
and denote by $1_{U_i}$ the indicator function for $U_i$.
The following lemma provides an estimate equivalent to \eqref{eq:1}
for the $(ax+b)$-group.
\begin{lemma}
  \label{lem:2}
  If $f\in L^2(G)$ is right differentiable up to order $2$ and 
  $R^\alpha f\in L^2(G)$ for $|\alpha|\leq 2$, then
  \begin{equation*}
    \left\| f - \sum_i f(x_i)1_{U_i} \right\|_{L^{2}}
    \leq C_\epsilon ( \| X_1 f\|_{L^{2}} 
    +  \| X_2 f\|_{L^{2}} + \| X_2 X_1 f\|_{L^{2}} )
  \end{equation*}
  where $C_\epsilon \to 0$ as $\epsilon \to 0$.
\end{lemma}

\begin{proof}
  For $x\in x_iU$ there are $s_1$ and $s_2$ in between $-\epsilon$ and
  $\epsilon$ such that $x_i = xe^{s_2X_2} e^{s_1X_1}$.
  Thus we get
  \begin{align*}
    |f(x)-f(x_i)|
    &= |f(x)-f(xe^{s_2X_2} e^{s_1X_1})| \\
    &\leq |f(x)-f(xe^{s_2X_2})|
    + |f(xe^{s_2X_2})-f(xe^{s_2X_2} e^{s_1X_1})| \\
    &= 
    \left| \int_{0}^{s_2} \frac{d}{dt_2} 
      f(x e^{t_2X_2}) \,dt_2  \right|
    + \left| \int_{0}^{s_1} \frac{d}{dt_1} 
      f(x e^{s_2X_2}e^{t_1X_1}) \,dt_1  \right| \\
    &\leq \int_{-\epsilon}^{\epsilon} \left| 
      X_2f(x e^{t_2X_2}) \right| \,dt_2
    + \int_{-\epsilon}^{\epsilon} \left| 
      X_1f(x e^{s_2X_2}e^{t_1X_1}) \right| \,dt_1
  \end{align*}
  Since
  $$
  e^{t_2X_2} e^{t_1X_1}
  = e^{t_1X_1}e^{t_2e^{-t_1}X_2}
  $$
  the term 
  $|X_1f(x e^{s_2X_2}e^{t_1X_1})|$ can be estimated by
  \begin{align*}
    | X_1 f(x e^{s_2X_2} e^{t_1X_1})| 
    &= | X_1 f(x e^{t_1X_1} e^{s_2e^{-t_1}X_2})| \\
    &\leq | X_1 f(x e^{t_1X_1}e^{s_2e^{-t_1}X_2}) 
    - X_1 f(x e^{t_1X_1} )| +  |X_1 f(x e^{t_1X_1} )| \\
    &=
    \left| \int_{0}^{s_1} 
    \frac{d}{dt_2} X_1 f(x e^{t_1X_1} e^{t_2e^{-t_1}X_2}) \,dt_2 \right|
    +  |X_1 f(x e^{t_1X_1})| \\
    &=
    \left| \int_{0}^{s_1} 
    e^{-t_1}X_2 X_1 
    f(x e^{t_1X_1} e^{t_2e^{-t_1}X_2}) \,dt_2 \right|
    +  |X_1 f(x e^{t_1X_1})| \\
    &=
    \left| \int_{0}^{s_1} 
    e^{-t_1}X_2 X_1 
    f(x e^{t_2 X_2}e^{t_1X_1}) \,dt_2 \right|
    +  |X_1 f(x e^{t_1X_1})| \\
    &\leq   
    e^{\epsilon}
    \int_{-\epsilon}^{\epsilon} 
    \left|X_2 X_1 f(x e^{t_2X_2} e^{t_1X_1})\right| 
    \,dt_2
    +  |X_1 f(x e^{t_1X_1})| \\
  \end{align*}  
  We therefore obtain the following estimate for $|f(x)-f(x_i)|$:
  \begin{align*}
    |f(x)-f(x_i)|
    &\leq \int_{-\epsilon}^{\epsilon} \left| 
      X_2f(x e^{t_2X_2}) \right| \,dt_2
    + \int_{-\epsilon}^{\epsilon} 
    |X_1 f(x e^{t_1X_1})| \,dt_1 \\
    &\qquad+ e^{\epsilon} 
    \int_{-\epsilon}^{\epsilon} \int_{-\epsilon}^{\epsilon} 
    \left|X_2 X_1 f(x e^{t_2X_2} e^{t_1X_1})\right| 
    \,dt_2\,dt_1
  \end{align*}
  This expression no longer depends on $i$ and using 
  Fubini's theorem to change the order of integration
  (or Minkowski to move the $L^{2}$-norm inside
  the integrals over $t_1$ and $t_2$) we get
  \begin{align*}
   \left\| f - \sum_i f(x_i)1_{U_i} \right\|_{L^{2}}
   &\leq
   \left\|
     \int_{-\epsilon}^{\epsilon}  
       r_ {e^{t_2X_2}} |X_2f|  \,dt_2
   \right\|_{L^{2}} \\
   &\qquad+ 
   \left\|
     \int_{-\epsilon}^{\epsilon} 
     r_ {e^{t_1X_1}} |X_1 f| \,dt_1
   \right\|_{L^{2}} \\
   &\qquad+
   e^{\epsilon} \left\|
    \int_{-\epsilon}^{\epsilon} \int_{-\epsilon}^{\epsilon} 
    r_ {e^{t_2X_2} e^{t_1X_1}} |X_2 X_1 f | 
    \,dt_2\,dt_1
   \right\|_{L^{2}} \\
   &\leq
   \int_{-\epsilon}^{\epsilon}  
   \| r_ {e^{t_2X_2}} X_2f \|_{L^2}  \,dt_2\\
   &\qquad+ 
     \int_{-\epsilon}^{\epsilon} 
     \| r_ {e^{t_1X_1}} X_1 f\|_{L^{2}} \,dt_1
    \\
   &\qquad+
   e^{\epsilon} 
    \int_{-\epsilon}^{\epsilon} \int_{-\epsilon}^{\epsilon} 
    \| r_ {e^{t_2X_2} e^{t_1X_1}} X_2 X_1 f \|_{L^{2}}
    \,dt_2\,dt_1 \\
   &\leq C_\epsilon ( \| X_1 f\|_{L^{2}} +
   \| X_2 f\|_{L^{2}} 
   +  \| X_2X_1 f\|_{L^{2}} )
 \end{align*}
 where we have also used that right translation inside
 $L^{2}(G)$ is continuous. Note that by the above calculations
 $C_\epsilon \to 0$ when $\epsilon \to 0$.


\end{proof}

We are now able to obtain the following sampling result
for reproducing kernel subspaces of $L^2(G)$.
\begin{theorem}
  \label{thm:1}
  If right differentiation up to order $2$
  is continuous on $H_\phi =L^2(G)*\phi$, then
  we can choose points $x_i$ such that
  the norms $\| f\|_{L^2}$ and $\| \{ f(x_i)  \}\|_{\ell^2}$
  are equivalent. Thus $\ell_{x_i}\phi^\vee$ forms a frame for $H_\phi$.
\end{theorem}

\begin{proof}
  First we note that if right differentiation is continuous
  then by Lemma~\ref{lem:2} there is a constant $C_\epsilon$
  such that
  \begin{equation*}
    \left\| f - \sum_i f(x_i)1_{U_i} \right\|_{L^{2}}
    \leq C_\epsilon  \| f\|_{L^{2}} 
  \end{equation*}
  By \cite{Pesenson2000} it is possible to choose
  $\epsilon$ and $x_i$ such that the sets
  $x_iU_{\epsilon/4}$ are disjoint and $C_\epsilon < 1$.
  Further note that as shown in \cite{Feichtinger1989a} both 
  $\| \sum_i \lambda_i 1_{x_iU_{\epsilon/4}}\|_{L^2}$ and
  $\| \sum_i \lambda_i 1_{x_iU_{\epsilon}}\|_{L^2}$ 
  define equivalent norms on $\ell^2$.
  Since
  $1_{x_iU_{\epsilon/4}} \leq 1_{U_i}$ we get
  \begin{align*}
    \| \{ f(x_i)\}\|_{\ell^2}
    &\leq C \| \sum_i f(x_i) 1_{x_iU_{\epsilon/4}}\|_{L^2}  \\
    &\leq C \| \sum_i (f-f(x_i)) 1_{x_iU_{\epsilon/4}}\|_{L^2} 
    + C \| \sum_i f 1_{x_iU_{\epsilon/4}}\|_{L^2}  \\
    &\leq C \| f -  \sum_i f(x_i) 1_{U_i} \|_{L^2} 
    + C \| \sum_i f 1_{U_i} \|_{L^2} \\
    &\leq C (1+C_\epsilon) \| f  \|_{L^2}
  \end{align*}
  Similarly
  we have
  \begin{align*}
    \| f \|_{L^2}
    &\leq  \| \sum_i (f-f(x_i)) 1_{U_i} \|_{L^2} 
    +  \| \sum_i f(x_i)  1_{U_i} \|_{L^2} \\
    &\leq C_\epsilon \| f\|_{L^2} + 
    \| \sum_i f(x_i) 1_{x_iU_\epsilon} \|_{L^2} \\
    &\leq C_\epsilon \| f\|_{L^2} + 
    D \| \{ f(x_i) \} \|_{\ell^2} \\
  \end{align*}
  Since $C_\epsilon < 1$ the inequalities above can be combined into
  \begin{equation*}
    \frac{1-C_\epsilon}{D} \| f\|_{L^2}
    \leq \| \{ f(x_i)\} \|_{\ell^2}
    \leq C (1+C_\epsilon) \| f \|_{L^2}
  \end{equation*}
  which shows the equivalence of the norms.

  Since $$f(x_i) = \int f(y)\phi^\vee (x_i^{-1}y)\,dy$$
  the vectors $\ell_{x_i}\phi^\vee$ form a frame.
\end{proof}

This leaves of course task of showing the continuity
of left differentiation on $H_\phi$. In 
the following example most statements are already proven
in \cite{Holschneider1995}.

\subsubsection{Example: Wavelet transform}

Let $\pi$ be the irreducible unitary representation of
$G$ on the space 
\begin{equation*}
  L^2_+ = \{ f\in L^2(\mathbb{R})  
  \mid \mathrm{supp}(\widehat{f}) \subseteq [0,\infty) \}
\end{equation*}
given by
\begin{equation*}
  \pi(a,b) f(x) = \frac{1}{\sqrt{a}} f\left( \frac{x-b}{a} \right)
\end{equation*}
Define the subspace
\begin{equation*}
  \mathcal{S}_+ = \{ f\in \mathcal{S}(\mathbb{R})  
  \mid \mathrm{supp}(\widehat{f}) \subseteq [0,\infty) \}
  \subseteq L^2_+
\end{equation*}
then $\mathcal{S}_+$ is invariant under the differential
operators
\begin{equation*}
  \pi(X)u = \lim_{t\to 0} \frac{\pi(e^{tX})u-u}{t}
\end{equation*}
for $X\in \mathfrak{g}$.
Denote by $\pi(R^\alpha)$ the differential operator
\begin{equation*}
  \pi(R^\alpha) u =  \pi(X_{\alpha(k)})\pi(X_{\alpha(k-1)})\cdots \pi(X_{\alpha(1)}) u
\end{equation*}
which also leaves $\mathcal{S}_+$ invariant.
For a non-zero $u\in \mathcal{S}_+$ define the wavelet transform
of $f\in L^2_+$ by
\begin{equation*}
  W_u(f)(a,b) = \ip{f}{\pi(a,b)u} = \int f(x) \frac{1}{\sqrt{a}}
  \overline{u\left( \frac{x-b}{a} \right)}\, dx
\end{equation*}
We can normalize $u$ according to the Duflo-Moore theorem \cite{Duflo1976}
such that for all $f\in L^2_+$
\begin{equation*}
  W_u(f)*W_u(u) = W_u(f) 
\end{equation*}
Then the wavelet tranform $W_u$ becomes an isometric 
isomorphism from $L^2_+$
into the reproducing kernel Hilbert space $L^2(G)*W_u(u)$.
Also the functions $W_u(f)$ is smooth and for $X\in \mathfrak{g}$ we have
\begin{equation*}
  R^\alpha W_u(f) = W_{\pi(R^\alpha)u}(f)
\end{equation*}
The Duflo-Moore theorem again tells us that
\begin{equation*}
  W_{\pi(R^\alpha)u}(f) = W_u(f)*W_{\pi(R^\alpha)u}(u)
\end{equation*}
Since $\pi(R^\alpha)u \in \mathcal{S}_+$ we have
$W_{\pi(R^\alpha)u}(u)\in L^1(G)$
as shown in \cite{Holschneider1995}
and therefore
$$\| W_{\pi(R^\alpha)u}(f)\|_{L^2} \leq C_\alpha \| W_u(f)\|_{L^2}.$$
This shows that the continuities in Theorem~\ref{thm:1}
are satisfied and we can
reconstruct $W_u(f)$ from it samples. In other words we
can reconstruct $f\in L^2_+$ from
sampling its wavelet coefficients.

\section{Sampling in reproducing kernel Banach spaces on Lie groups}
\label{sec:srkbs}
A Banach space of functions is called a reproducing kernel
Banach space if point evaluation is continuous. 
We restrict our attention to reproducing kernel Banach 
spaces where the reproducing formula is given by 
a Lie group convolution.
We derive local oscillation estimates for
such spaces and use them to give a discrete
characterization of the reproducing kernel space. 
In particular we obtain frame and atomic decompositions
for reproducing kernel Banach spaces under certain smoothness conditions
on the kernel.

\subsection{Reproducing kernel Banach spaces}
Let $G$ be a locally compact group 
with left invariant Haar measure $\mu$. 
Denote by $\ell_x$ and $r_x$ the left and right
translations given by
\begin{equation*}
  \ell_x f(y) = f(x^{-1}y)
  \qquad\text{and}\qquad
  r_xf(y) = f(yx)
\end{equation*}
A Banach space of functions is called left or right invariant
if there is a constant $C_x$ such that
\begin{equation*}
  \| \ell_x f\|_B \leq C_x \| f\|_B
  \qquad\text{or}\qquad\| r_x f\|_B \leq C_x \| f\|_B
\end{equation*}
respectively.
We will always assume that for compact $U$ there is a 
constant $C_U$ such that for all $f\in B$
\begin{equation}
  \label{eq:3}
  \sup_{y\in U} \| \ell_yf \|_B \leq C_U \| f\|_B
  \qquad\text{and}\qquad
  \sup_{y\in U} \| r_yf \|_B \leq C_U \| f\|_B
\end{equation}
For $1\leq p <\infty$ the space $L^p(G)$ denotes the equivalence 
class of measurable functions for which
\begin{equation*}
  \| f\|_{L^p} = \left( \int |f(x)|^p \,d\mu(x)\right)^{1/p} < \infty
\end{equation*}
We will often
write $dx$ instead of $d\mu(x)$.
The space $L^\infty(G)$ consists of equivalence classes of
measurable functions for which
\begin{equation*}
  \| f \|_{L^\infty} = \mathrm{ess\,sup}_{x\in G} |f(x)|
\end{equation*}
The spaces $L^p(G)$ are left and right invariant and satisfy \eqref{eq:3}
for $1\leq p\leq \infty$, however
the left and right translations are only continuous for
$1\leq p <\infty$.

When $f,g$ are measurable functions on $G$ for which 
the product $f(x)g(x^{-1}y)$ is integrable for all $y\in G$
we define the convolution $f*g$ as
\begin{equation*}
  f*g(y) = \int_G f(x)g(x^{-1}y)\,d\mu(x)
\end{equation*}
A Banach space of functions $B$ is called solid if 
$|f| \leq |g|$ and $g\in B$ imply that $f \in B$.
All spaces $L^p(G)$ are solid, but Sobolev subspaces are not.

In this article we will only work with reproducing kernel 
Banach spaces which can be constructed in the following
manner.
Let $B$ be a solid left invariant
Banach space of functions on $G$ which satisfies \eqref{eq:3}.
Assume that there
is a non-zero $\phi\in B$ such that 
\begin{equation*}
  \left| \int_G f(y)\phi(y^{-1})\,dy\right|
  \leq C \| f\|_B
\end{equation*}
By the left invariance of $B$ the convolution 
$f*\phi$ is well-defined by
\begin{equation*}
  f*\phi(x) = \int_G f(y)\phi(y^{-1}x)\,dy
\end{equation*}
Assume that $\phi$ satisfies the reproducing formula
\begin{equation*}
  \phi*\phi = \phi
\end{equation*}
and that convergence in $B$ implies convergence (locally) in measure.
Then the space
\begin{equation*}
  B_\phi = \{ f\in B\mid f=f*\phi   \}
\end{equation*}
is a non-zero reproducing kernel Banach subspace of $B$. 
Let us for completeness prove this statement 
(though it is already contained in \cite{Christensen2010} in disguise) 
by showing that $B_\phi$ is a closed subspace of $B$. 
Let $f_n \in B_\phi$ be a sequence 
converging to $f\in B$ then $f_n$ converges locally in measure.
Therefore there is a subsequence $f_{n_k}$ converging to
$f$ almost everywhere. Thus for almost all $x\in G$ we have
\begin{align*}
  |f(x)  - f*\phi(x)|
  &\leq
  |f(x)  - f_{n_k}(x)| 
  + |f_{n_k}(x) - f_{n_k}*\phi(x)|
  + | f_{n_k}*\phi(x) -  f*\phi(x)| \\
  &\leq
    |f(x)  - f_{n_k}(x)| 
  + |f_{n_k}(x) - f_{n_k}*\phi(x)|
  + C \| f_{n_k} -  f\|_{B}
\end{align*}
The middle term is $0$ and the two remaining terms can be
made arbitrarily small so $f=f*\phi$ which shows that
$B_\phi$ is closed in $B$. 
Point evaluation is continuous for $f\in B_\phi$ 
by the left invariance of $B$:
\begin{equation*}
  |f(x)| 
  \leq C \|\ell_{x^{-1}}f \|_B
  \leq C_x \| f\|_B
\end{equation*}

The discretizations we will investigate on $B_\phi$ can be
thought of as Riemann sums for the reproducing formula 
\begin{equation*}
  f(x) = \int_G f(y)\phi(y^{-1}x)\,dy
\end{equation*}
which holds for all $f\in B_\phi$.

\subsection{Atomic decompositions and Banach frames}
We will derive atomic decompositions and frames for 
reproducing kernel Banach spaces, and here we introduce the two notions.
Further we introduce sequence spaces and partitions of unity 
used to obtain the discrete characterizations.

\begin{definition}
  \label{def:1}
  Let $B$ be a Banach space and $B^\#$ an associated Banach sequence 
  space with index set $I$. If for $\lambda_i\in B^*$ and $\phi_i \in B$ 
  we have
  \begin{enumerate}
  \item $\{\lambda_i(f) \}_{i\in I} \in B^\#$ for all $f\in B$
  \item the norms 
    $\| \lambda_i(f)\|_{B^\#}$ and $\| f\|_B$ are equivalent
  \item $f$ can be written $f = \sum_i \lambda_i(f) \phi_i$
  \end{enumerate}
  then $\{(\lambda_i,\phi_i) \}$ is an atomic decomposition
  of $B$ with respect to $B^\#$.
\end{definition}

More generally a Banach frame for a Banach space can be defined as
\begin{definition}
  \label{def:3}
  Let $B$ be a Banach space and $B^\#$ an associated Banach sequence 
  space with index set $I$. If for $\lambda_i\in B^*$
  we have
  \begin{enumerate}
  \item $\{\lambda_i(f) \}_{i\in I} \in B^\#$ for all $f\in B$
  \item the norms 
    $\| \lambda_i(f)\|_{B^\#}$ and $\| f\|_B$ are equivalent
  \item there is a bounded reconstruction operator $S:B^\# \to B$ 
    such that $S(\{\lambda_i(f) \}) = f$
  \end{enumerate}
  then $\{ \lambda_i \}$ is an Banach frame for 
  $B$ with respect to $B^\#$.
\end{definition}

In Hilbert spaces the existence of the operator $S$ is automatic given
the equivalence of the norms $\| \lambda_i(f)\|_{B^\#}$ and $\| f\|_B$.
Further, the operator $S$ has been shown to exist 
for $p$-frames for reproducing kernel Banach spaces in \cite{Han2009}. 
For general Banach spaces this is not the case as is demonstrated in
\cite{Casazza2005}.

We will work with very particular
Banach sequence spaces which are constructed from
a solid Banach function space $B$. These spaces were introduced in
\cite{Feichtinger1989a}.
For a compact neighbourhood $U$ of the identity we call
the sequence $\{ x_i\}_{i\in I}$ $U$-relatively separated
if $G\subseteq \cup_i x_iU$ and there is an $N$ such that
\begin{equation*}
  \sup_i (\# \{ j \mid x_iU\cap x_jU \neq \emptyset \}) \leq N
\end{equation*}
For a $U$-relatively separated sequence $X=\{ x_i\}_{i\in I}$ 
define the space $B^\#(X)$ to be the collection
of sequences $\{ \lambda_i\}_{i\in I}$ for which
\begin{equation*}
  \sum_{i\in I} |\lambda_i| 1_{x_iU} \in B
\end{equation*}
when the sum is taken to be pointwise.
If the compactly supported continuous functions are dense in $B$ then
this sum also converges in norm.
Equipped with the norm 
$$\| \{\lambda_i \}\|_{B^\#} = \|
\sum_{i\in I} |\lambda_i| 1_{x_iU} \|_B$$
this is a solid Banach sequence space.
In the case were $B=L^p(G)$ we get that $B^\#(X)=\ell^p(I)$.
For fixed $X=\{ x_i\}_{i\in I}$ the space $B^\#(X)$ only depends
on the compact neighbourhood $U$ up to norm equivalence.
Further, if $X=\{ x_i \}_{i\in I}$ and $Y=\{ y_i \}_{i\in I}$ are 
two $U$-relatively separated sequences 
with same index set such that $x_i^{-1}y_i \in V$
for some compact set $V$, then $B^\#(X)=B^\#(Y)$ equivalent norms.
For these properties consult Lemma 3.5 in \cite{Feichtinger1989a}.

Given a compact neighbourhood $U$ of the identity
the non-negative functions $\psi_i$ are called a bounded uniform 
partition of unity subordinate to $U$ (or $U$-BUPU), if
there is a $U$-relatively separated 
sequence $\{ x_i\}$, such that 
$\mathrm{supp}(\psi_i)\subseteq x_i U$ and
$\sum_i \psi_i =1$. A partition of unity could consist of indicator
functions, however on a Lie group $G$ it is possible to find smooth
$U$-BUPU's whenever $U$ is contained in a ball of radius less 
than the injectivity radius of $G$ 
(see for example \cite[Lemma 2.1]{Pesenson2000}). 
For the existence of $U$-BUPU's with elements contained in 
homogeneous Banach spaces see the paper \cite{Feichtinger1981}.

\subsection{Local oscillation estimates on Lie groups}

Let $G$ be Lie group with Lie algebra $\mathfrak{g}$
with exponential function $\exp:\mathfrak{g}\to G$. Then for
$X\in\mathfrak{g}$ we define the right and left
differential operators (if the limits exist)
\begin{equation*}
  R(X)f(x) = \lim_{t\to 0} \frac{f(x\exp(tX))-f(x)}{t}
  \qquad\text{and}\qquad
  L(X)f(x) = \lim_{t\to 0} \frac{f(\exp(tX)x)-f(x)}{t}
\end{equation*}
Fix a basis $\{ X_i\}_{i=1}^{\mathrm{dim}(G)}$ 
for $\mathfrak{g}$. For a multi index $\alpha$ of
length $|\alpha| =k$ 
with entries between $1$ and $\mathrm{dim}(G)$ 
we introduce the operator $R^\alpha$ of subsequent right differentiations
\begin{equation*}
  R^{\alpha} f  
  = R(X_{\alpha(k)})  R(X_{\alpha(k-1)}) \cdots R(X_{\alpha(1)}) f
\end{equation*}
Similarly we introduce the operator $L^\alpha$ of subsequent 
left differentiations
\begin{equation*}
  L^{\alpha} f  
  = L(X_{\alpha(k)})  L(X_{\alpha(k-1)}) \cdots L(X_{\alpha(1)}) f
\end{equation*}
We call $f$ right (or left) differentiable of order $n$ 
if for every $x$ and all $|\alpha|\leq n$ the derivatives
$R^\alpha f(x)$ (or $L^\alpha f(x)$) exist.

In the following we will often use this lemma
\begin{lemma}\label{lem:1}
  Let $U$ be a compact set and fix a basis element $X_k\in\mathfrak{g}$.
  There is a constant $C_U$ such that for any $y\in U$
  and $|s|\leq \epsilon$ 
  \begin{enumerate}
  \item 
    If $f$ is right differentiable of order $1$, then 
      $$ |f(xe^{sX_k}y) - f(xy)|
    \leq C_U \int_{-\epsilon}^{\epsilon}
    \sum_{n=1}^{\mathrm{dim}(G)} | R(X_n) f(xe^{tX_k}y)|\, dt$$
  \item 
    If $f$ is left differentiable of order $1$, then
      $$|f(ye^{sX_k}x) - f(yx)|
    \leq C_U \int_{-\epsilon}^{\epsilon}
    \sum_{n=1}^{\mathrm{dim}(G)} | L(X_n) f(ye^{tX_k}x)|\, dt$$
  \end{enumerate}
  The constant $C_U$ depends only on $U$ and
  $C_{U'} \leq C_U$ for $U' \subseteq U$.
\end{lemma}
\begin{proof}
  First use the fundamental theorem of calculus to get
  \begin{align*}
    |f(xe^{sX_k}y) - f(xy)|
    &= |f(xye^{sAd_{y^{-1}}(X_k)}) - f(xy)| \\
    &= \left| 
      \int_0^s \frac{d}{dt} f(xye^{tAd_{y^{-1}}(X_k)} \,dt
    \right| \\
    &\leq \int_0^s \left| \frac{d}{dt} f(xye^{tAd_{y^{-1}}(X_k)}\right| \,dt \\
    &
    \leq \int_{-\epsilon}^{\epsilon}
    |R(Ad_{y^{-1}}(X_k)) f(xe^{tX_k}y)|\, dt \\
  \end{align*}
  The adjoint representation can be written as
  \begin{equation*}
    Ad_{y^{-1}}(X_k)) = c_1(y)X_1 + \dots + c_n(y)X_n
  \end{equation*}
  where the coefficients $c_i$ depend continuously on $y$
  (and also depend on $X_k$).
  So for smooth $f$ we have the pointwise inequality
  \begin{equation*}
    |R(Ad_{y^{-1}}(X_k))f| 
    \leq |c_1(y)||R(X_1)f| + \dots + |c_n(y)||R(X_n)f|
    \leq C_U(X_k) \sum_{n=1}^{\mathrm{dim}(G)} |R(X_n)f|
  \end{equation*}
  where $C_U(X_k)$ is $\max_{y\in U} |c_i(y)|$.
  Let $C_U = \max_{k} |C_U(X_k)|$, then we obtain
  \begin{align*}
    |f(xe^{sX_k}y) - f(xy)|
    &\leq C_U \int_{-\epsilon}^{\epsilon}
    \sum_{n=1}^{\mathrm{dim}(G)} | R(X_n) f(xye^{tAd_{y^{-1}}(X_k)})|\, dt \\
    &= C_U \int_{-\epsilon}^{\epsilon}
    \sum_{n=1}^{\mathrm{dim}(G)} | R(X_n) f(xe^{tX_k}y)|\, dt
  \end{align*}
  From the definition of $C_U$ above it follows that
  $C_{U'} \leq C_U$ for $U' \subseteq U$.
  A similar argument works for left differentiations.
\end{proof}

From now on we let $U_\epsilon$ denote the set
\begin{equation*}
  U_\epsilon = \left\{ \prod_{k=1}^n \exp(t_kX_{k}) 
    \Mid -\epsilon \leq t_k \leq \epsilon \right\}.
\end{equation*}
We further define the right local oscillations
\begin{equation*}
  M_r^\epsilon f(x) = \sup_{u\in U_\epsilon} |f(xu^{-1})-f(x)|
\end{equation*}
and the left local oscillations
\begin{equation*}
  M_l^\epsilon f(x) = \sup_{u\in U_\epsilon} |f(ux)-f(x)|
\end{equation*}

For formulating the next lemma we need some notation. By $\delta$ we 
denote an $n$-tuple $\delta = (\delta_1,\dots,\delta_n)$ with
$\delta_i \in \{0,1 \}$. The length $|\delta|$ of $\delta$ is the number
of non-zero entries $| \delta | = \delta_1+\dots +\delta_n$.
Further the function $\tau_\delta$ is defined as
\begin{equation*}
  \tau_\delta (t_1,\dots,t_n)
  = e^{\delta_1t_1X_1} \dots e^{\delta_nt_nX_n}
\end{equation*}

\begin{lemma}\label{lem:5}
  If $f$ is right differentiable of order $n=\mathrm{dim}(G)$
  there is a constant $C_\epsilon $ such that
  \begin{equation*}
    M_r^\epsilon f(x) \leq C_\epsilon 
    \sum_{1\leq |\alpha|\leq n} 
    \sum_{|\delta|=|\alpha|}
    \underbrace{\int_{-\epsilon}^\epsilon  
      \cdots \int_{-\epsilon}^\epsilon}_{\text{$|\delta|$ integrals}}
    |R^\alpha f(x\tau_\delta(t_1,\dots,t_n)^{-1})|
    (dt_1)^{\delta_1} \cdots (dt_n)^{\delta_n}
  \end{equation*}
  If $f$ is left differentiable of order $n=\mathrm{dim}(G)$
  there is a constant $C_\epsilon $ such that
  \begin{equation*}
    M_l^\epsilon f(x) \leq C_\epsilon 
    \sum_{1\leq |\alpha|\leq n} 
    \sum_{|\delta|=|\alpha|}
    \underbrace{\int_{-\epsilon}^\epsilon  
      \cdots \int_{-\epsilon}^\epsilon}_{\text{$|\delta|$ integrals}}
    |L^\alpha f(\tau_\delta(t_1,\dots,t_n)x)|
    (dt_1)^{\delta_1} \cdots (dt_n)^{\delta_n}
  \end{equation*}  
  For $\epsilon' \leq \epsilon$ we have $C_{\epsilon'} \leq C_\epsilon$.
\end{lemma}

\begin{proof}
  For any $x$ there is an
  element $\sigma =e^{s_nX_n}\dots e^{s_1X_1} \in U_\epsilon ^{-1}$
  such that 
  \begin{equation*}
    M_r^\epsilon f(x) 
    = |f(xe^{s_nX_n}\dots e^{s_1X_1})-f(x)|
  \end{equation*}
  Denote by $\sigma_k$ the element in $U_\epsilon^{-1}$ given by
  \begin{equation*}
    \sigma_k = e^{s_nX_n}e^{s_{n-1}X_{n-1}}\dots e^{s_{k}X_{k}}
  \end{equation*}
  with the convention that $\sigma_{n+1}=e$.
  The elements $\sigma_k$ depend on $x$, and we wish to estimate
  the function $M_r^\epsilon f(x) = 
  |f(x\sigma_1)-f(x)|$ by an expression without any $\sigma_k$. 
  To do so we make repeated use of the fundamental theorem of
  calculus in form of the previous lemma.
  
  For any $n$-tuple $\delta$ of $0$'s and $1$'s 
  and for a smooth function $f$ we define
  \begin{equation*}
    T_{\alpha,\delta} f(x)
    = \underbrace{\int_{-\epsilon}^\epsilon 
      \dots \int_{-\epsilon}^\epsilon}_{\text{$|\delta|$ integrals}}
    |R^{\alpha} f(x\tau_\delta(t_1,\dots,t_n)^{-1})|\, 
    (dt_1)^{\delta_1}    (dt_2)^{\delta_2}\dots (dt_n)^{\delta_n}
  \end{equation*}
  We first show that if $\delta_m=0$ for $m\geq k$, then
  \begin{equation*}
    T_{\alpha,\delta} f(x\sigma_k) 
    \leq
    T_{\alpha,\delta} f(x\sigma_{k+1}) + 
    C_{U_\epsilon^{-1}} 
    \sum_{|\alpha'|=|\delta'|=|\alpha|+1} T_{\alpha',\delta' }f(x\sigma_{k+1}) 
  \end{equation*}
  where $\delta'_m=0$ for $m\geq k+1$.
  To show this note that
  \begin{align*}
    |R^{\alpha} &f(x\sigma_k \tau_\delta(t_1,\dots,t_n)^{-1})| \\
    &\leq |R^{\alpha} f(x\sigma_{k+1}\tau_\delta(t_1,\dots,t_n)^{-1})|\\
    &\qquad+ 
    |R^{\alpha} f(x\sigma_{k+1}e^{s_kX_k}\tau_\delta(t_1,\dots,t_n)^{-1}) -
    R^{\alpha} f(x\sigma_{k+1}\tau_\alpha(t_1,\dots,t_n)^{-1})| \\
    &\leq |R^{\alpha} f(x\sigma_{k+1}\tau_\delta(t_1,\dots,t_n)^{-1})| \\
    &\qquad+    C_{U_\epsilon^{-1}} \int_{-\epsilon}^{\epsilon}
    \sum_{n=1}^{\mathrm{dim}(G)} 
    | R(X_n) R^{\alpha}f(x\sigma_{k+1}e^{t_kX_k}
    \tau_\delta(t_1,\dots,t_n)^{-1})|\, dt_k \\
  \end{align*}
  The terms 
  $R(X_n) R^{\alpha}f(x\sigma_{k+1}e^{t_kX_k}
  \tau_\delta(t_1,\dots,t_n)^{-1})$ 
  are of the type 
  \begin{equation*}
    R^{\alpha'} f(x \sigma_{k+1} \tau_{\delta'}(t_1,\dots,t_n)^{-1})
  \end{equation*}
  with $|\alpha'| = |\alpha|+1$ and $\delta'_m = 0$ for $m\geq k+1$.
  Therefore 
  \begin{align*}
    T_{\alpha,\delta} &f(x\sigma_k) \\
    &= \int_{-\epsilon}^\epsilon 
    \dots \int_{-\epsilon}^\epsilon
    |R^{\alpha} f(x\tau_\delta(t_1,\dots,t_n)^{-1})|\, 
    (dt_1)^{\delta_1}    (dt_2)^{\delta_2}\dots (dt_n)^{\delta_n}\\
    &\leq 
    \int_{-\epsilon}^\epsilon 
    \dots \int_{-\epsilon}^\epsilon
    |R^{\alpha} f(x\sigma_{k+1}\tau_\delta(t_1,\dots,t_n)^{-1})|\, 
    (dt_1)^{\delta_1}    (dt_2)^{\delta_2}\dots (dt_n)^{\delta_n}\\
    &\qquad + 
    C_{U_\epsilon^{-1}} \underbrace{\int_{-\epsilon}^\epsilon 
      \dots \int_{-\epsilon}^\epsilon}_{\text{$|\alpha|+1$ integrals}}
    \sum_{m=1}^n 
    | R(X_m) R^{\alpha}f(x\sigma_{k+1}e^{t_kX_k}
    \tau_\delta(t_1,\dots,t_n)^{-1})|\, dt_k
    \,(dt_1)^{\delta_1}    (dt_2)^{\delta_2}\dots (dt_n)^{\delta_n}\\
    &\leq T_{\alpha,\beta}f(x\sigma_{k+1})
    + C_{U_\epsilon^{-1}} \sum_{|\alpha'|=|\delta'|=|\alpha|+1}
    T_{\alpha',\delta'}f(x\sigma_{k+1})
  \end{align*}
  The assumption that $\delta_m=0$ for $m\geq k$ ensures that
  each $\tau_{\delta'}$ is in $U_\epsilon$ and thus the constant
  $C_{U_\epsilon^{-1}}$ shows up in the application of the previous 
  lemma. As $\epsilon$ is chosen smaller this constant is thus bounded.
    
  Estimating the right local oscillation we first obtain 
  \begin{align*}
    M_r^\epsilon f(x) 
    &= |f(x\sigma_1)-f(x)| \\
    &\leq \sum_{l=1}^{n} |f(x\sigma_{l})-f(x\sigma_{l+1})| \\
    &= \sum_{l=1}^{n} |f(x\sigma_{l+1}e^{s_lX_l})-f(x\sigma_{l+1})| \\
    &\leq \sum_{l=1}^{n} \int_{-\epsilon}^{\epsilon}
    | R(X_l) f(x\sigma_{l+1}e^{t_lX_l})|\, dt_l \\
  \end{align*}
  This is a finite sum of terms of the type
  $T_{\alpha,\delta} f(x\sigma_k)$
  with $|\alpha|=|\delta|=1$ and $2\leq k\leq n+1$.
  Each of the terms with $2\leq k\leq n$ can in turn 
  be estimated by a sum of terms $T_{\alpha,\delta} f(x\sigma_k)$
  with $1\leq |\alpha|=|\delta|\leq 2$ for $3\leq k\leq n+1$. 
  Repeating these steps we find
  \begin{equation*}
    M_r^\epsilon 
    f(x) \leq C_\epsilon 
    \sum_{1\leq |\alpha|=|\beta|\leq n} T_{\alpha,\beta}f(x)
  \end{equation*}
  where $C_\epsilon$ is a constant for which
  $C_{\epsilon'}\leq C_\epsilon$ when $\epsilon'\leq \epsilon$.
  
  The inequality for the left local oscillation is obtained
  analogously.
\end{proof}

The following local oscillation estimate will be of great
importance to our sampling results.
\begin{theorem}\label{thm:3}
    If $f\in B$ is right differentiable up to order $n=\mathrm{dim}(G)$
    and the derivatives $R^{\alpha} f$ are in $B$ for
    $1\leq |\alpha| \leq n $,
    then 
    \begin{equation*}
      \| M_r^\epsilon f \|_B \leq C_\epsilon \sum_{1\leq |\alpha| \leq n }
      \| R^{\alpha} f\|_B
    \end{equation*}
    Here $C_\epsilon\to 0$ as $\epsilon\to 0$.
\end{theorem}

\begin{proof}
  As in the proof of the previous lemma let
  \begin{equation*}
    T_{\alpha,\delta} f(x)
    = \underbrace{\int_{-\epsilon}^\epsilon 
      \dots \int_{-\epsilon}^\epsilon}_{\text{$|\delta|$ integrals}}
    |R^{\alpha} f(x\tau_\delta(t_1,\dots,t_n)^{-1})|\, 
    (dt_1)^{\delta_1}    (dt_2)^{\delta_2}\dots (dt_n)^{\delta_n}
  \end{equation*}
  We now show that there is a constant $C$ (only depending on $U$ and $B$)
  such that 
  \begin{equation*}
    \| T_{\alpha,\delta}f \|_{B} \leq C (2\epsilon)^{|\delta|} 
    \| R^{\alpha} f\|_B
  \end{equation*}
  For this we use the Minkowski inequality to get
  \begin{align*}
    \| T_{\alpha,\delta}f \|_{B}
    &= \left\| \int_{-\epsilon}^\epsilon 
      \dots \int_{-\epsilon}^\epsilon
    |r_{\tau_\delta(t_1,\dots,t_n)^{-1}} R^{\alpha}  f |\, 
    (dt_1)^{\delta_1}    (dt_2)^{\delta_2}\dots (dt_n)^{\delta_n}
    \right\|_B \\
    &\leq 
    \int_{-\epsilon}^\epsilon 
      \dots \int_{-\epsilon}^\epsilon
    \|r_{\tau_\delta(t_1,\dots,t_n)^{-1}} R^{\alpha}  f \|_B \, 
    (dt_1)^{\delta_1}    (dt_2)^{\delta_2}\dots (dt_n)^{\delta_n}
  \end{align*}
  According to \eqref{eq:3} let 
  $C$ be the smallest constant such that for all $f\in B$ and
  for all $u\in U_\epsilon$ we have $\| r_{u^{-1}} f\|_B \leq C \| f\|_B$.
  Then 
  \begin{align*}
    \| T_{\alpha,\delta}f \|_{B}
    &\leq 
    \int_{-\epsilon}^\epsilon 
      \dots \int_{-\epsilon}^\epsilon
    C \| R^{\alpha}  f \|_B \, 
    (dt_1)^{\delta_1}    (dt_2)^{\delta_2}\dots (dt_n)^{\delta_n}
    \leq C (2\epsilon)^{|\delta|}  \| R^{\alpha}  f \|_B \, 
  \end{align*}
  Since $M_r^\epsilon f$ can be estimated by a finite sum of terms of the type
  $T_{\alpha,\delta}f$ with $|\delta| \geq 1$
  the triangle inequality can be used to finish
  the proof.
\end{proof}

\begin{corollary}
  \label{cor:1}
  If the functions in $B_\phi$ are smooth and the mappings
  $$
  B_\phi \ni f \mapsto R^{\alpha}f \in B
  $$
  are continuous for $|\alpha|\leq \mathrm{dim}(G)$, then
  there is a $C_\epsilon$ such that 
  \begin{equation*}
    \| M_r^\epsilon f \|_B \leq C_\epsilon \| f\|_B
  \end{equation*}
  with $C_\epsilon \to 0$ as $\epsilon\to 0$.
\end{corollary}

It is typically not hard to show that if $f\in B_\phi$ then
\begin{equation*}
  R^\alpha f = f*R^\alpha \phi
\end{equation*}
Thus we need only check that convolution with $R^\alpha \phi$ is
continuous. In the following case we will use differentiability of
the kernel to obtain estimates of 
$M_r^\epsilon f$ for $f\in B_\phi$.
This ties our results with \cite{Grochenig1991},
though we do not require the kernel to be integrable.
The previous results are more general and
are in particular very easy to verify for band-limited functions, 
whereas the following theorem is harder to
apply in that case (the author is at present not aware of
an application of this theorem to band-limited functions).

\begin{theorem}\label{thm:4}
  If the mappings 
  $$B_\phi \ni f\mapsto |f|*|R^\alpha\phi|\in B$$ 
  are continuous for $|\alpha|\leq \mathrm{dim}(G)$ then
  there is a $C_\epsilon$ such that
  \begin{equation*}
    \| M_r^\epsilon f\|_B \leq C_\epsilon \| f\|_B 
  \end{equation*}
  with $C_\epsilon \to 0$ as $\epsilon\to 0$.
\end{theorem}

\begin{proof}
  For $f\in B_\phi$ we have that
  \begin{align*}
    M_r^\epsilon f(x)
    &= \sup_{u\in U_\epsilon} |f(xu^{-1}) - f(x)| \\
    &= \sup_{u\in U_\epsilon} 
    \left| 
      \int f(y)[\phi(y^{-1}xu^{-1}) - \phi(y^{-1}x)] \,dy \right| \\
    &\leq 
    \int |f(y)|M_r^\epsilon\phi(y^{-1}x) \,dy 
  \end{align*}
  
  Since $\phi$ is differentiable we know that
  $M_r^\epsilon\phi(y^{-1}x)$ can be estimated by
  \begin{equation*}
    M_r^\epsilon \phi(x) \leq C_\epsilon 
    \sum_{1\leq |\alpha|\leq n} 
    \sum_{|\delta|=|\alpha|}
    \underbrace{\int_{-\epsilon}^\epsilon  
      \cdots \int_{-\epsilon}^\epsilon}_{\text{$|\delta|$ integrals}}
    |R^\alpha \phi(x\tau_\delta(t_1,\dots,t_n)^{-1})|
    (dt_1)^{\delta_1} \cdots (dt_n)^{\delta_n}
  \end{equation*}
  Thus the assumption that the convolutions 
  $|f|*|R^\alpha \phi|$ are continuous from $B_\phi$ to $B$
  and the right invariance
  of $B$ can be used to finish the proof.
\end{proof}

\subsection{Atomic decompositions and frames for 
  reproducing kernel Banach spaces}
In this section we will derive sampling theorems and atomic decompositions
for the reproducing kernel Banach space $B_\phi$. The results are 
similar to those in  
\cite{Feichtinger1988,Feichtinger1989a,Feichtinger1989b,Grochenig1991} and
more recently \cite{Rauhut2005} and \cite{Nashed2010}, but unlike
these references we do not require integrability of the reproducing kernel.

The following sampling theorem can be utilized together with
the result of Corollary~\ref{cor:1} and Theorem~\ref{thm:4}. 

\begin{theorem}\label{thm:9}
  Assume there is a $C_\epsilon$ such that $C_\epsilon\to 0$ for
  $\epsilon\to 0$ such that for all $f\in B_\phi$ 
  we have $\| M_r^\epsilon f\|_B \leq C_\epsilon\| f\|_{B}$.
  We can choose $\epsilon$ small enough
  that for every $U_\epsilon$-relatively separated set
    $\{ x_i\}$ the norms $\| \{ f(x_i) \}\|_{B^\#}$
    and $\| f\|_B$ are equivalent. 
\end{theorem}

\begin{remark}\label{rem:1}
  We would like to note that 
  Theorem~\ref{thm:9} is sufficient to prove that
  sampling provides a Banach frame in the case $B=L^p(G)$ according
  to \cite[Theorem 3.1]{Han2009}. 
  Thus we are able to obtain
  sampling theorems for cases where the convolution with the kernel
  is not a continuous projection.
\end{remark}

\begin{proof}
  Choose $\epsilon$ small enough that $C_\epsilon<1$ 
  and let $\{ x_i \}$ be $U_\epsilon$-relatively separated. 
  Note that there is an $N$ such that each
  $x_iU_\epsilon$ overlap with at most $N$ other $x_jU_\epsilon$.
  The following calculation shows that $\{ f(x_i)\} \in B^\#$.
  \begin{align*}
    \sum_i |f(x_i)| 1_{x_iU_\epsilon}(x)
    &\leq \sum_i |f(x)- f(x_i)| 1_{x_iU_\epsilon}(x) 
    + \sum_i |f(x)| 1_{x_iU_\epsilon}(x) \\
    &\leq \sum_i M_r^\epsilon f(x) 1_{x_iU_\epsilon}(x) 
    + \sum_i |f(x)| 1_{x_iU_\epsilon}(x) \\
    &\leq N(M_r^\epsilon f(x) + |f(x)|)
  \end{align*}
  Both the functions $M_r^\epsilon f$ and $|f|$ are in $B$ by
  assumption and thus
  \begin{align*}
    \| \{f(x_i) \}\|_{B^\#}
    = \left\| \sum_i |f(x_i)| 1_{x_iU_\epsilon}\right\|_{B}
    \leq N (C_\epsilon +1) \| f\|_B
  \end{align*}
  
  We now show that  
  $(1-C_\epsilon) \| f\|_B \leq \| \{f(x_i) \}\|_{B^\#}$.
  Let $\psi_i$ be a $U_\epsilon$-uniform bounded partition of unity,
  i.e. $\mathrm{supp}(\psi_i)\subseteq x_iU_\epsilon$ and 
  $\sum_i \psi_i =1$ a.e. 
  \begin{align*}
    |f(x)| 
    &= \sum_i |f(x)|\psi_i(x)  \\
    &\leq \sum_i |f(x)-f(x_i)|\psi_i(x) + \sum_i |f(x_i)|\psi_i(x) \\
    &\leq M_r^\epsilon f(x) + \sum_i |f(x_i)| 1_{x_iU_\epsilon}(x) \\
  \end{align*}
  Therefore 
  \begin{equation*}
    \| f\|_B 
    \leq \| M_r^\epsilon f\|_B + \|\{ f(x_i)\} \|_B^{\#}
    \leq C_\epsilon \| f\|_B + \|\{ f(x_i)\} \|_B^{\#}
  \end{equation*}
  By assumption $C_\epsilon < 1$ so we obtain
  \begin{equation*}
    (1-C_\epsilon) \| f\|_B \leq \|\{ f(x_i)\} \|_B^{\#}
  \end{equation*}
  This finishes the proof.
\end{proof}

The following theorem provides a reconstruction operator in the
case where convolution with $\phi$ is continuous on $B$. 

\begin{theorem}\label{thm:8}
  Assume there is a $C_\epsilon$ such that $C_\epsilon\to 0$ for
  $\epsilon\to 0$ such that for all $f\in B_\phi$ 
  we have $\| M_r^\epsilon f\|_B \leq C_\epsilon\| f\|_{B}$.
  If convolution with $\phi$ is continuous
  on $B$, then
  we can choose $\epsilon$ small enough
  that for any $U_\epsilon$-BUPU $\{ \psi_i\}$
  with $\mathrm{supp}(\psi_i)\subseteq x_iU_\epsilon$ 
  the operator $T_1:B_\phi \to B_\phi$ given by
  \begin{equation*}
    T_1f = \sum_i f(x_i)(\psi_i*\phi)
  \end{equation*}
  is invertible. 
  The convergence of the sum is pointwise, and
  if $C_c(G)$ is dense in $B$ then
  the convergence is also in norm.
\end{theorem}

\begin{proof}
  We have that
  \begin{equation*}
    \left| f(x) - \sum_i f(x_i)\psi_i(x) \right|
    \leq \sum_i |f(x)-f(x_i)|\psi_i(x)
    \leq \sum_i M_r^\epsilon f(x)\psi_i(x)
    = M_r^\epsilon f(x)
  \end{equation*}
  so the solidity of $B$ ensures that
  \begin{equation*}
    \left\| 
      f - \sum f(x_i)\psi_i
    \right\|_B
    \leq C_\epsilon \| f\|_B
  \end{equation*}
  The continuity of convolution with the reproducing kernel gives
  \begin{equation*}
    \left\| 
      f - \left( \sum f(x_i)\psi_i\right)*\phi
    \right\|_{B_\phi}
    \leq C_\epsilon \| f\|_{B_\phi}
  \end{equation*}
  
  Lastly we show that the operator 
  \begin{equation*}
    \left( \sum f(x_i)\psi_i\right)*\phi
  \end{equation*}
  is indeed the operator $T$.
  To do so we use the dominated convergence theorem
  to swap the sum and the integral in
  \begin{equation*}
    \int 
    \left( \sum_{i} f(x_i) \psi_i(x)\right)\phi(x^{-1}y)
    \,dx
  \end{equation*}
  The sum $\sum_{i} f(x_i) \psi_i(x)$
  is to be understood as the pointwise limit of partial sums.
  Any partial sum $F_P(x) = \sum_{i\in P} f(x_i) \psi_i(x)$
  is dominated by
  \begin{align*}
    |F_P(x)|
    &\leq \sum_{i\in P} |f(x_i)| \psi_i(x) \\
    &\leq \sum_{i\in P} |f(x_i)-f(x)| \psi_i(x) 
    + \sum_{i\in P} |f(x)| \psi_i(x) \\
    &\leq \sum_{i\in P} M_r^\epsilon f(x) \psi_i(x) 
    + \sum_{i\in P} |f(x)| \psi_i(x) \\
    &\leq M_r^\epsilon f(x) + |f(x)|
  \end{align*}
  Both $M_r^\epsilon f$ and $|f|$ are in $B$ and therefore (by our
  assumptions on $\phi$) the integrable function
  $(M_r^\epsilon f(x)+|f(x)|)|\phi(x^{-1}y)|$ dominates
  $|F_N(x)\phi(x^{-1}y)|$. This allows the sum
  and integral to be swapped to get
  \begin{equation*}
    \left\| 
      f - \sum f(x_i)(\psi_i*\phi)
    \right\|_{B_\phi}
    \leq C_\epsilon \| f\|_{B_\phi}
  \end{equation*}
  Choosing $\epsilon$ small enough that $C_\epsilon<1$ the operator $T$ can
  be inverted using its Neumann series. 

  In \cite{Rauhut2005} it has been shown that if the compactly supported
  continuous functions are dense, then the sum converges in norm.
  We therefore skip that part of the proof.
\end{proof}

The previous result in conjunction with Corollary~\ref{cor:1} 
and Theorem~\ref{thm:4} only requires continuity involving left
differentiation. We will now state results that also involve
right differentation. 

\begin{lemma}\label{lem:3}
  Assume the convolutions 
  $f\mapsto f*|L^\alpha\phi|$ are continuous $B \to B$ for 
  $|\alpha|\leq \mathrm{dim}(G)$, then
  the mapping
  \begin{equation*}
    B^\# \ni \{ \lambda_i \} \mapsto \sum_i \lambda_i \ell_{x_i}\phi\in B_\phi
  \end{equation*}
  is continuous.
\end{lemma}

\begin{proof}
  Let $\{ \lambda_i\} \in B^\#$ and define
  \begin{equation*}
    f(x) = \sum_i \lambda_i \ell_{x_i}\phi(x)
  \end{equation*}
  with pointwise convergence of the sum.
  We will show that this defines a function in $B$.
  For every $x$ we have
  \begin{align*}
    |f(x)| &\leq
    \sum_i  |\lambda_i| |\phi(x_i^{-1}x)| \\
    &= \mu(U)^{-1}  \int \sum_i  |\lambda_i| 1_{x_iU}(y) |\phi(x_i^{-1}x)| \, dy \\
    &\leq \mu(U)^{-1} \left( \int \sum_i 
      |\lambda_i| 1_{x_iU}(y) |\phi(y^{-1}x)- \phi(x_i^{-1}x)| \, dy  
      + \int \sum_i 
      |\lambda_i| 1_{x_iU}(y) |\phi(y^{-1}x)| \, dy \right)\\
    &\leq \mu(U)^{-1} \left( \int \sum_i 
      |\lambda_i| 1_{x_iU}(y) M_\ell^\epsilon \phi(y^{-1}x) \, dy 
      + \int \sum_i 
      |\lambda_i| 1_{x_iU}(y) |\phi(y^{-1}x)| \, dy\right) \\
    &\leq     \mu(U)^{-1} \left( \sum_i  |\lambda_i| 1_{x_iU})\right)* 
    (M_\ell^\epsilon \phi + |\phi|) (x)
  \end{align*}
  The function $F=\sum_i |\lambda_i| 1_{x_iU}(y)$ is in $B$ and 
  our assumptions ensure that the functions $F*|\phi|$ and
  $F*M_\ell^\epsilon \phi$ are also in $B$. The solidity of $B$ thus 
  ensures that the function $f$ is in $B$.
  
  We will now show that $f$ is reproduced by convolution with $\phi$.
  Note that any partial sum $f_N = \sum_{i=1}^N \lambda_i \ell_{x_i}\phi$
  is in $B$ and us reproduced by convolution by $\phi$.
  We have to show that 
  $f(x)=\lim_{N\to \infty} f_N(x)$ 
  is also reproduced by convolution with  $\phi$.
  The calculation above 
  shows that any partial sum $f_N$ is dominated by the function
  \begin{equation*}
    G = \mu(U)^{-1} \left( \sum_i  |\lambda_i| 1_{x_iU})\right)* 
    (M_\ell^\epsilon \phi + |\phi|) \in B
  \end{equation*}
  Thus we have $|f_N(y)\phi(y^{-1}x)|\leq |G(y)\phi(y^{-1}x)|$ and
  the dominated convergence theorem gives
  \begin{equation*}
    f*\phi(x)
    =  (\lim_{N\to \infty} f_N)*\phi(x) 
    = \lim_{N\to \infty} (f_N*\phi)(x) 
    = \lim_{N\to \infty} f_N(x) 
    = f(x).
  \end{equation*}
  The continuity of the mapping follows from the calculations above.
\end{proof}

\begin{theorem}\label{thm:5}
  Assume that the convolutions
  $f\mapsto |f|*|L^\alpha\phi|$ are continuous $B \to B$ for 
  $|\alpha|\leq \mathrm{dim}(G)$.
  We can choose $\epsilon$ 
  and $U_\epsilon$-relatively separated
  points $\{ x_i\}$, such that for any
  $U_\epsilon$-BUPU $\{\psi_i \}$ with 
  $\mathrm{supp}(\psi_i) \subseteq x_iU_\epsilon$
  the operator (we let $\lambda_i(f) = \int f(x) \psi_i(x)\, dx$)
  \begin{equation*}
    T_2f = \sum_i \lambda_i(f) \ell_{x_i}\phi
  \end{equation*}
  is invertible on $B_\phi$.
  The convergence of the sum is pointwise, and
  if $C_c(G)$ is dense in $B$ then
  the convergence is also in norm.
  Further $\{ \lambda_i(T^{-1}_2 f),\ell_{x_i}\phi\}$ is an atomic decomposition
  for $B_\phi$.
\end{theorem}

\begin{proof}
  For $f\in B_\phi$ we have the following estimate
  \begin{align*}
    \sum_i |\lambda_i(f)| 1_{x_iU}(y)
    &\leq \sum_i |\lambda_i| 1_{x_iU}(y) \\
    &\leq \sum_i \left| \int f(x)\psi_i(x)\, dx \right| 1_{x_iU}(y) \\
    &\leq \sum_i \int |f(x)| 1_{x_iU}(x)\, dx 1_{x_iU}(y) \\
    &= \sum_i \int |f(x)| 1_{x_iU}(x)\, dx 1_{x_iU}(y)
  \end{align*}
  For each $y$ only $N$ of the neighbourhoods $x_iU$ overlap, and also
  $1_{x_iU}(x)1_{x_iU}(y) \leq 1_{U^{-1}U}(x^{-1}y)$ so
  we get
  \begin{align*}
    \sum_i |\lambda_i| 1_{x_iU}(y)
    \leq N \int |f(x)|1_{U^{-1}U}(x^{-1}y)\, dx
    = N |f|*1_{U^{-1}U}(y)
  \end{align*}
  The function $|f|*1_{U^{-1}U}$ is in $B$ with norm
  bounded by $C\| f\|_B$ for some constant $C$ 
  (in the sense of Bochner integrals).
  Therefore the sequence $\lambda_i(f)$ is in $B^\#$ with
  norm estimated by
  \begin{equation*}
    \| \{ \lambda_i(f)  \} \|_{B^\#}
    \leq CN \| f \|_{B}
  \end{equation*}

  By Lemma~\ref{lem:3} we see that $T_2f\in B_\phi$ and 
  \begin{align*}
    |f(x) - T_2f(x)|
    &= \left| 
      f(x) - \sum_i \int f(y)\psi_i(y)\,dy \phi(x_i^{-1}x)
    \right| \\
    &=  \left| 
      \int f(y)\phi(y^{-1}x)\,dy 
      - \sum_i \int f(y)\psi_i(y)\,dy \phi(x_i^{-1}x)
    \right|  \\
    &=  \left| 
      \int \sum_i \psi_i(y) f(y)\phi(y^{-1}x)\,dy 
      - \sum_i \int \psi_i(y) f(y)\psi_i(y)\,dy \phi(x_i^{-1}x)
    \right| \\
    &\leq  
    \sum_i \left| \int \psi_i(y) f(y)\phi(y^{-1}x)\,dy 
      - \int \psi_i(y) f(y)\psi_i(y)\,dy \phi(x_i^{-1}x)
    \right| \\
    &=
    \sum_i \left| \int \psi_i(y) f(y)(\phi(y^{-1}x) -
      \phi(x_i^{-1}x))\,dy
    \right| \\
    &\leq
    \sum_i \int \psi_i(y) |f(y)| M_l^\epsilon \phi(y^{-1}x) \,dy  \\
    &=\int |f(y)| M_l^\epsilon \phi(y^{-1}x) \,dy  \\
    &= |f|*M_l^\epsilon \phi(x)
  \end{align*}
  The continuity of the mappings
  \begin{equation*}
    B\ni f\mapsto f*|L^{\alpha}\phi| \in B
  \end{equation*}
  and Lemma~\ref{lem:5}
  ensure that $T_2 f\in B$ and $T_2$ is well-defined. 

  We will now show that $\{ \ell_{x_i}\phi\}$
  yields an atomic decomposition.

  Since $T_2^{-1}f\in B_\phi$ we therefore have
  \begin{equation*}
    \| \{ \lambda_i(T_2^{-1}f)  \} \|_{B^\#}
    \leq C \| T_2^{-1}f \|_{B}
    \leq C \| f \|_{B}
  \end{equation*}
  Further the reconstruction formula
  $f= \sum_i  \lambda(T_2^{-1}f) \ell_{x_i}\phi$ Lemma~\ref{lem:3}
  tells us that
  \begin{equation*}
    \| f\|_{B} \leq C     \| \{ \lambda_i(T_2^{-1}f)  \} \|_{B^\#}
  \end{equation*}
  thus showing that we have an atomic decomposition.
\end{proof}

\begin{theorem}\label{thm:6}
  Assume the convolutions $f\mapsto f*|L^\alpha\phi|$ and 
  $f\mapsto f*|R^\alpha\phi|$ are continuous $B \to B$ for   
  $|\alpha|\leq \mathrm{dim}(G)$.
  Then we can choose $\epsilon$ 
  and $U_\epsilon$-relatively separated
  points $\{ x_i\}$ such that for any
  $U_\epsilon$-BUPU $\{\psi_i \}$ with 
  $\mathrm{supp}(\psi_i) \subseteq x_iU_\epsilon$
  the operator given by (we let $c_i = \int \psi_i(x)\,dx$)
  \begin{equation*}
    T_3f = \sum_i c_i f(x_i)\ell_{x_i}\phi
  \end{equation*}
  is invertible.
  The convergence of the sum is pointwise, and
  if $C_c(G)$ is dense in $B$ then
  the convergence is also in norm.
  Further $\{ c_iT_3^{-1}f(x_i),\ell_{x_i}\phi\}$ 
  is an atomic decomposition for $B_\phi$ and $\{c_i\ell_{x_i}\phi \}$ is
  a frame.
\end{theorem}

\begin{proof}
  Since 
  $$
  c_i 
  = \int \psi_i(x)\, dx
  \leq \int 1_{x_iU}(x) \,dx
  = \mu(U)
  $$
  we have
  \begin{align*}
    \left| \sum_i c_i f(x_i)1_{x_iU}(x)\right|
    &\leq  
    \mu(U) \sum_i [|f(x_i)-f(x)| + |f(x)|]1_{x_iU}(x) \\
    &\leq \mu(U) \sum_i [M_r^\epsilon f(x) + |f(x)|]1_{x_iU}(x) \\
    &\leq \mu(U) N [M_r^\epsilon f(x) + |f(x)|]
  \end{align*}
  Therefore Theorem~\ref{thm:4} gives
  \begin{equation*}
    \| \{ c_if(x_i) \} \|_{B^\#} \leq 
    \mu(U)N (C_\epsilon +1) \| f\|_B
  \end{equation*}
  Thus $T_3f\in B_\phi$ by Lemma~\ref{lem:3} and
  \begin{align*}
    |f(x) - T_3f(x)|
    &\leq 
    \left| f(x) - \sum_i \left( \int f(y)\psi_i(y) \,dy\right)
      \ell_{x_i}\phi(x) \right| \\
    &\qquad+ \left| \sum_i \left( \int f(y)\psi_i(y) \,dy\right)
      \ell_{x_i}\phi(x) - \sum_i c_i f(x_i)\ell_{x_i}\phi(x) \right|
  \end{align*}
  The first expression was estimated in the previous theorem, so we
  concentrate on
  \begin{align*}
    \left| \sum_i \left( \int f(y)\psi_i(y) \,dy\right)
      \ell_{x_i}\phi\right. &- \left.\sum_i c_i f(x_i)\ell_{x_i}\phi \right| \\
    &\leq
    \sum_i \left( \int |f(y)-f(x_i)||\phi(x_i^{-1}x)|\psi_i(y)
      \,dy\right) \\
    &\leq
    \sum_i \int |f(y)-f(x_i)||\phi(y^{-1}x)| \psi_i(y)
      \,dy \\
    &\qquad + \sum_i \int |f(y)-f(x_i)||\phi(y^{-1}x)-\phi(x_i^{-1}x)|
      \psi_i(y) \,dy \\
    &\leq
    \sum_i \left( \int Mf(y)|\phi(y^{-1}x)| \psi_i(y) \,dy\right) \\
    &\qquad+ \sum_i \int Mf(y)M_r^\epsilon \phi(y^{-1}x)
      \psi_i(y) \,dy \\
    &= M_l^\epsilon f*|\phi|(x) + M_l^\epsilon f*M_r^\epsilon\phi(x)
  \end{align*}
  Now, by our assumptions the functions in the last expression
  are all in $B$,
  and their norms are dominated by the norms of convolution with of
  $f$ with $|L^\alpha \phi|$ and $|R^\alpha\phi|$ for $|\alpha|\leq
  \mathrm{Dim}(G)$.
  Therefore 
  \begin{equation*}
    \| f - T_3f \|_B \leq C_\epsilon \| f\|_B
  \end{equation*}
  where $C_\epsilon\to 0$ as $\epsilon \to 0$.

  Since $T_3^{-1}f \in B_\phi$ we thus have
  \begin{equation*}
    \| \{ c_iT_3^{-1}f(x_i) \} \|_{B^\#} \leq 
    C \| T_3^{-1}f\|_B
    \leq C \| f\|_B
  \end{equation*}
  Any $f\in B_\phi$ can be written
  \begin{equation*}
    f = \sum_i c_i T_3^{-1}f(x_i)\ell_{x_i}\phi
  \end{equation*}
  and thus by Lemma~\ref{lem:3} the norm of $f$ satisfies
  \begin{equation*}
    \| f\|_B \leq C \| \{ c_iT_3^{-1}f(x_i)\}\|_{B^\#}
  \end{equation*}
  which finishes the proof that $\{
  c_iT_3^{-1}f(x_i),\ell_{x_i}\phi\}$ 
  forms an 
  atomic decomposition of $B_\phi$.
\end{proof}

\section{Coorbit Spaces on Lie groups}
\noindent
In this section we apply the sampling theorems from section~\ref{sec:srkbs}
to a certain class of reproducing spaces. These spaces are
images of Banach spaces of distributions (so-called coorbit spaces)
under a wavelet transform. Thus we yield sampling theorems for
a large class of Banach spaces including modulation spaces, Besov
spaces, Bergman spaces and Hilbert spaces of band-limited functions.
Similar sampling theorems are known for spaces related to irreducible
and integrable representations 
\cite{Feichtinger1989a,Feichtinger1989b,Grochenig1991}.
We replace the integrability condition with smoothness arguments
which also apply to non-integrable and non-irreducible cases.

\subsection{Construction of coorbit spaces}

Let $S$ be a Fr\'echet space and let $S^\cdual$ be the 
conjugate linear dual equipped with the weak$^*$ topology. 
We assume that $S$ is
continuously imbedded and 
weakly dense in $S^\cdual$. The conjugate
dual pairing of elements $v\in S$ and $v'\in S^\cdual$ will be denoted
by $\dup{v'}{v}$. 
As usual define the contragradient
representation $(\pi^\cdual,S^\cdual)$ by
\begin{equation*}
  \dup{\pi^\cdual(x)v'}{v}=\dup{v'}{\pi(x^{-1})v}.
\end{equation*}
Then $\pi^*$ is a continuous representation of $G$ on $S^\cdual$. 
For a fixed vector 
$u\in S$ define the linear map $W_u:S^*\to C(G)$ by
\begin{equation*}
  W_u(v')(x) = \dup{v'}{\pi(x)u}.
\end{equation*}
The map $W_u$ is called \emph{the voice transform} or 
\emph{the wavelet transform}.

In \cite{Christensen2010} we listed minimal conditions ensuring that
spaces of the form
\begin{equation*}
  \Co_S^u B = \{ v' \in S^\cdual | W_u(v') \in B \}
\end{equation*}
equipped with the norm $\| v' \| = \| W_u(v')\|_B$ are
$\pi^*$ invariant Banach spaces.
The space $\Co_S^u B$ is called the coorbit space of $B$ related
to $u$ and $S$.
\begin{assumption}\label{as:1}
  Assume there is a non-zero cyclic vector $u\in S$ 
  satisfying the following properties
  \begin{enumerate}
    \renewcommand{\labelenumi}{(R\arabic{enumi})}
    \renewcommand{\theenumi}{(R\arabic{enumi})}
  \item the reproducing formula $W_{u}(v)*W_{u}(u)=W_{u}(v)$ is true
    for all $v\in S$ \label{r1}
  \item 
    the mapping
    $Y\ni F\mapsto \int_G F(x)W_u(u)(x^{-1})\,dx \in \mathbb{C}$
    is continuous \label{r2}
  \item if $F=F*W_u(u)\in Y$ then the mapping
    $S \ni v \mapsto \int F(x) \dup{\pi^\cdual(x)u}{v} \,dx
    \in\mathbb{C}$ 
    is in $S^\cdual$\label{r3}
  \item the mapping 
    $S^\cdual\ni \phi \mapsto \int \dup{\phi}{\pi(x)u} 
    \dup{\pi^\cdual(x)u}{u}  \,dx \in\mathbb{C}$ 
    is weakly continuous \label{r4}
  \end{enumerate}
\end{assumption}
A vector $u$ satisfying Assumption~\ref{as:1} 
is called an \emph{analyzing
vector}. 

The subspace $B_u$ of $B$ defined by 
$$B_u = \{ F\in B | F=F*W_u(u) \},$$
is a reproducing kernel Banach space.
By \cite{Christensen2010} it follows that
$\Co_S^u B$ is $W_u:\Co_S^u B \to B_u$ is an isometric isomorphism
intertwining $\pi^*$ and left translation.

\subsection{Sampling of wavelet transform}

We now list conditions ensuring that we can obtain the frame
inequality from Theorem~\ref{thm:9}. 
A vector $u\in S$ is called weakly differentiable in the direction
$X\in\mathfrak{g}$ if there is a vector denoted $\pi(X)u \in S$
such that for all $v'\in S^*$ 
\begin{equation*}
  \dup{v'}{\pi(X)u}=
  \frac{d}{dt}\Big|_{t=0}\dup{v'}{\pi(e^{tX})u}
\end{equation*}
For the differential operators $R^\alpha$ we write
$\pi(R^\alpha)u$ for a vector which satisfies
\begin{equation*}
  \dup{v'}{\pi(R^\alpha)u}=
  \dup{v'}{\pi(X_{\alpha(k)}) \pi(X_{\alpha(k-1)}) \cdots \pi(X_{\alpha(1)})u}
\end{equation*}
We use the notation $\pi(R^\alpha)$ for the differential
operators on $S$ because they match the right differential operators 
$R^\alpha$ on $B$: 
if $f(x) = W_u(v)(x)$, then $R^\alpha f(x) = W_{\pi(R^\alpha)u}(v)(x)$.
\begin{assumption}\label{as:2}
  Assume there is a non-zero cyclic vector $u\in S$ 
  satisfying Assumption~\ref{as:1}. Further assume
  that $u$ is weakly differentiable up to order
  $\mathrm{dim}(G)$ and that
  \begin{enumerate}
    \renewcommand{\labelenumi}{(D\arabic{enumi})}
    \renewcommand{\theenumi}{(D\arabic{enumi})}
  \item there are non-zero constants $c_\alpha$
    such that $W_{u}(v)*W_{\pi(R^\alpha)u}(u)=c_\alpha W_{\pi(R^\alpha)u}(v)$
    for all $v\in S$ \label{d1}
  \item the mapping 
    $S^\cdual\ni \phi \mapsto \int \dup{\phi}{\pi(x)u} 
    \dup{\pi^\cdual(x)u}{\pi(R^\alpha)u}  \,dx \in\mathbb{C}$ 
    is weakly continuous \label{d2}
  \item the mappings
    $B_u\ni F\mapsto F*W_{\pi(R^\alpha)u}(u)\in B$
    are continuous for all $|\alpha|\leq \mathrm{dim}(G)$ \label{d3}
  \end{enumerate}
\end{assumption}

\begin{remark}
  Notice that for $\alpha=0$ the properties \ref{d1} 
  and \ref{d2} correspond to 
  \ref{r1} and \ref{r4} respectively. The condition \ref{d2} is used to 
  extend the convolution relation from \ref{d1} to all $v\in S^*$.
\end{remark}

\begin{theorem}
  \label{thm:7}
  If $u\in S$ satisfies Assumption~\ref{as:2} then
  we can choose $\epsilon$ small enough that for any
  $U_\epsilon$-relatively separated set $\{ x_i\}$ there
  are $0 < A_1\leq A_2<\infty$ such that
  \begin{equation*}
    A_1 \| v'\|_{\Co_S^u B} 
    \leq \| \{ \dup{v'}{\pi(x_i)u} \}\|_{B^\#}
    \leq A_2 \| v'\|_{\Co_S^u B} 
  \end{equation*}
  If convolution with $W_u(u)$ is continuous on $B$, then
  $\pi(x_i)u$ is a frame for $\Co_S^u B$ with reconstruction operator
  \begin{equation*}
    v' =
    W_u^{-1 }T_1^{-1} \left( \sum_i W_u(v')(x_i)\psi_i*W_u(u) \right)
  \end{equation*}
  where $\{\psi_i \}$ is any $U_\epsilon$-BUPU for which
  $\supp(\psi_i)\subseteq x_iU_\epsilon$.
\end{theorem}

\begin{proof}
  Let us first show that \ref{d1} and \ref{d2} ensure that
  \begin{equation*}
    W_{u}(v')*W_{\pi(R^\alpha)u}(u)=c_\alpha W_{\pi(R^\alpha)u}(v')    
  \end{equation*}
  for $v'\in S^*$.
  Let $v_\beta$ be a net in $S$ converging to $v'$. Then
  \begin{align*}
    c_\alpha W_{\pi(R^\alpha)u}(v')(x)
    &= \lim_\beta c_\alpha W_{\pi(R^\alpha)u}(v_\beta) (x) \\
    &= \lim_\beta W_{u}(v_\beta)*W_{\pi(R^\alpha)u}(u) (x) \\
    &= \lim_\beta  
    \int \dup{v_\beta}{\pi(xy)u} \dup{\pi^*(y)u}{\pi(R^\alpha)u} \,dy \\
    &= \lim_\beta  
    \int \dup{\pi^*(x^{-1})v_\beta}{\pi(y)u} \dup{\pi^*(y)u}{\pi(R^\alpha)u} \,dy \\
    &= 
    \int \dup{\pi^*(x^{-1})v'}{\pi(y)u} \dup{\pi^*(y)u}{\pi(R^\alpha)u} \,dy \\
    &= 
    \int \dup{v'}{\pi(y)u} \dup{u}{\pi(y^{-1}x)\pi(R^\alpha)u} \,dy \\
    &= W_{u}(v')*W_{\pi(R^\alpha)u}(u) (x) \\
  \end{align*}
  Therefore, if $v'\in \Co_S^u B$ we have
  \begin{equation*}
    W_{\pi(R^\alpha)u}(v') = \frac{1}{c_\alpha} W_{u}(v')*W_{\pi(R^\alpha)u}(u)
  \end{equation*}
  and the continuity requirement \ref{d3} ensures
  that $W_{\pi(R^\alpha)u}(v')\in B$ and
  \begin{equation*}
    \| W_{\pi(R^\alpha)u}(v')\|_B     \leq C_\alpha \| W_u(v')\|_B
  \end{equation*}
  By Theorem~\ref{thm:3} there is a constant
  $C_\epsilon$ such that
  \begin{equation*}
    \| M_r^\epsilon W_u(v')\|_B 
    \leq C_\epsilon \| W_u(v')\|_B
  \end{equation*}
  and $C_\epsilon \to 0$ as $\epsilon \to 0$. 
  Theorem~\ref{thm:9} shows that
  there are $A_1,A_2$ such that
    \begin{equation*}
    A_1 \| v'\|_{\Co_S^u B} 
    \leq \| \{ \dup{v'}{\pi(x_i)u} \}\|_{B^\#}
    \leq A_2 \| v'\|_{\Co_S^u B} 
  \end{equation*}
  which proves the norm equivalence.
  If convolution with $W_u(u)$ is continuous on $B$ the
  reconstruction operator can be found using Theorem~\ref{thm:8}.
\end{proof}

\begin{remark}
  For $B=L^p(G)$ the sequence space
  is $B^\# = \ell^p$ and in this case a reconstruction operator
  is automatic when the frame inequality is given (see \cite{Han2009}). 
\end{remark}

\subsection{G{\aa}rding vectors and smooth representations}

In this section we will focus on square integrable group representations
and its smooth vectors. In particular we will show that
G{\aa}rding vectors are particularly nice to work with.

A unitary irreducible representation
$(\pi,H)$ is square integrable if there is a non-zero $u\in H$
such that the function $W_u(u)(x) = \ip{u}{\pi(x)u}$ is in $L^2(G)$.
Any such vector $u$ is called admissible. Duflo and Moore \cite{Duflo1976}
proved the following
\begin{theorem}[Duflo-Moore]
  \label{thm:11}
  If $(\pi,\mathcal{H})$ is square integrable, then there is a positive
  densely defined operator $C$ with domain $D(C)$ such that
  $W_u(u)$ is in $L^2(G)$ if and only if $u\in D(C)$.
  Furthermore, if $u_1,u_2 \in D(C)$ then
  \begin{equation*}
    \int_G (v_1,\pi(x)u_1)_H(\pi(x)u_2,v_2)_H \, dx
    = (Cu_2,Cu_1)_H (v_1,v_2)_H
  \end{equation*}
\end{theorem}
By choosing $u$ such that $\| Cu\|_H =1$ 
we automatically obtain a reproducing formula
\begin{equation*}
  W_u(v)*W_u(u) = W_u(v) 
\end{equation*}
for all $v\in H$.

A vector $v\in H$ is called smooth if the mapping
\begin{equation*}
  G\in x \mapsto \pi(x)v \in H
\end{equation*}
is smooth in the norm topology of $H$. 
The space of smooth vector is denoted $H_\pi^\infty$ and
is a Fr\'echet space when equipped with the seminorms
\begin{equation*}
  \| v\|_k = \sup_{|\alpha|\leq k} \| \pi(R^\alpha)v\|_{H}
\end{equation*}
For any $v\in H$ and any $f\in C_c^\infty(G)$
the vector $\pi(f)v$ defined by
\begin{equation*}
  \pi(f)v = \int f(x) \pi(x)v\,dx
\end{equation*}
is smooth and called a G{\aa}rding vector.

The following statement is an extension of 
a result found in \cite{Christensen2009}
without proof.
\begin{lemma}
  \label{lem:4}
  If $u\in H_\pi^\infty$ is in the domain of the operator
  $C$ from Theorem \ref{thm:11}, then the map
  \begin{equation*}
    H_\pi^{-\infty} \ni \phi \mapsto 
    \int \dup{\phi}{\pi(x)u} \dup{\pi(x)u}{v}  \,dx \in \mathbb{C}
  \end{equation*}
  is continuous in the weak topology for $v\in H_\pi^\infty$. Thus both \ref{r4} 
  and \ref{d2} are satisfied.
\end{lemma}

\begin{proof}
  For vectors $v$ in $H^\infty_\pi$ and $w\in H$ 
  the dual pairing $\dup{w}{v}$ is the inner product $\ip{w}{v}$ on $H$.
  For $v\in H_\pi^\infty$ we have
  \begin{equation*}
    H\in w\mapsto \int \ip{w}{\pi(x)u} \ip{\pi(x)u}{v}  \,dx
    = C_u\ip{w}{v}
  \end{equation*}
  and therefore the weakly defined vector
  \begin{equation*}
    \pi(W_u(v)^\vee)u 
    = \int  \ip{\pi(x)u}{v} \pi(x)u \,dx
= C_u v \in H_\pi^\infty
  \end{equation*}
  This proves the statement of the lemma.
\end{proof}

\begin{theorem}
  \label{thm:2}
  Let $(\pi,H)$ be a square integrable representation
  with smooth vectors $S=H_\pi^\infty$.
  Let $B$ be a left and right invariant 
  Banach function space and
  let $u\in S$ be 
  such that Assumption~\ref{as:1} is satisfied and
  further the mapping
  \begin{equation*}
    B\ni F \mapsto F*W_u(u)\in B
  \end{equation*}
  is continuous. Then 
  $\Co_{S}^u = \Co_{S}^{\widetilde{u}}$ for
  any (properly normalized) G{\aa}rding vector $\widetilde{u}$ and
  the vectors
  $\pi(x_i)\widetilde{u}$ form a Banach frame for both $\Co_{S}^{\widetilde{u}}$
  and $\Co_{S}^u$. Further $\pi(x_i)\widetilde{u}$ provide atomic decompositions
  for $\Co_{S}^u$ through Theorem~\ref{thm:5} and Theorem~\ref{thm:6}.
\end{theorem}

\begin{proof}
  Note, that if $f,g\in C_c^\infty(G)$ then
  admissible $u$ and $v\in H$ we have
  \begin{equation*}
    W_{\pi(f)u}(\pi(g)v)
    = g*W_u(u)*(f^\vee)
  \end{equation*}
  Since $g*L^2(G)*g^\vee \subseteq L^2(G)$ we thus see that
  any non-zero G{\aa}rding vector $\widetilde{u}=\pi(g)u$ is also
  admissible. Therefore we can normalize $\pi(g)u$ such
  that the reproducing formula 
  \begin{equation*}
    W_{\widetilde{u}}(v)*W_{\widetilde{u}}(\widetilde{u})
    =  W_{\widetilde{u}}(v)
  \end{equation*}
  is true. From now on let $\widetilde{u}$ be normalized accordingly.
  Further
  \begin{align*}
    \int_G F(x)W_{\pi(g)u}(v)(x^{-1})\, dx
    &= \int_G F(x)\ip{v}{\pi(x^{-1})\pi(g)u}\, dx \\
    &= \int_G \int_G F(x)\ip{v}{\pi(x^{-1})\pi(y)u}\overline{g(y)} \, dy \, dx \\
    &= \int_G \int_G F(yx)\ip{v}{\pi(x^{-1})u}\overline{g(y)} \, dy \, dx \\
    &= \int_G \int_G F(y^{-1}x)\ip{v}{\pi(x^{-1})u}\overline{g^\vee(y)} \, dy \, dx\\
    &= \int_G g^**F(x) W_{u}(v)(x)\, dx
  \end{align*}
  where $g^*(x) = \overline{g(y^{-1})}$.
  Since $g^**F \in Y$ 
  and depends continuously on $F$ (in the sense of Bochner integrals)
  the mapping 
  \begin{equation*}
    (F,v) \mapsto \int_G F(x)W_{\pi(g)u}(v)(x^{-1})\, dx
  \end{equation*}
  is continuous.
  This shows that
  $\Co_S^{\pi(g)u} B$ is a well-defined $\pi^*$-invariant Banach space.

  We now show that the norms on $\Co_S^{u} B$ and $\Co_S^{\pi(g)u} B$
  are equivalent. By the square integrability
  it follows that for any $v\in H$
  \begin{equation*}
    W_u(v)*W_{\widetilde{u}}(u) 
    = C_u W_{\widetilde{u}}(v) 
  \end{equation*}
  and
  \begin{equation*}
    W_{\widetilde{u}}(v)*W_{u}({\widetilde{u}})
    = C_{\widetilde{u}} W_{u}(v) 
  \end{equation*}
  These two formulas can be extended to all $v\in S^*=H_\pi^\infty$
  by Lemma~\ref{lem:4}.
  Since 
  \begin{align*}
    F\mapsto F*W_{u}(\pi(g)u)
    &= F*g*W_u(u)
    \\
    F\mapsto F*W_{\widetilde{u}}(u)
    &= F*W_u(u)*(g^\vee)
  \end{align*}
  are both continuous mappings
  it follows that 
  \begin{align*}
    \| W_u(v) \|_B
    = C \| W_{\widetilde{u}}(v)*W_{u}({\widetilde{u}})\|_B
    \leq C \| W_{\widetilde{u}}(v) \|_B
  \end{align*}
  and 
  \begin{align*}
    \| W_{\widetilde{u}}(v) \|_B
    = C \| W_{u}(v)* (g^\vee)\|_B
    \leq C \| W_u(v) \|_B
  \end{align*}
  Thus the norms on $\Co_S^{u} B$ and $\Co_S^{\pi(g)u} B$
  are equivalent.

  Finally we need to show that we can reconstruct $\phi\in \Co_S^{\pi(g)u} B$
  from it samples. For this it suffices to show that
  \begin{equation*}
    \| R^\alpha W_{\widetilde{u}}(\phi)\|_B \leq C \| W_{\widetilde{u}}(\phi)\|_B 
  \end{equation*}
  and apply Theorem~\ref{thm:3} and for example Theorem~\ref{thm:8}.
  Note, that 
  \begin{equation*}
    R^\alpha W_{\widetilde{u}}(\phi)
    = W_{u}(\phi)*(R^\alpha g)^\vee
  \end{equation*}
  By the continuity of convolution with $R^\alpha g$ we thus see
  that 
  $\| R^\alpha W_{\widetilde{u}}(\phi)\|_B 
  \leq C \| W_{u}(\phi) \|_B$
  and the previously proven norm equivalence gives
  $$\| R^\alpha W_{\widetilde{u}}(\phi)\|_B  
  \leq C \| W_{\widetilde{u}}(\phi)\|_B 
  $$
  to finish the proof that $\pi(x_i)\widetilde{u}$ is a frame. 
  The statements about atomic decompositions follow similarly.
\end{proof}

\begin{remark}
  \label{rem:3}
  Note that we need not necessarily work with the smooth vectors
  $H_\pi^\infty$. 
  In the coorbit theory introduced by Feicthinger and Gr\"ochenig
  \cite{Feichtinger1989a} the space
  $$S=H^1_w=\{v\in H \mid W_u(v)\in L^1_w \}$$ is used in the construction
  of coorbit. Here $w$ is a submultiplicative
  weight. In order to obtain sampling theorems they need to choose
  the analyzing vector in the space
  $$B_w = \{ u\in H \mid W_u(u),M_\epsilon^r W_u(u) \in L^1_w \}$$
  For any $u$ with $W_u(u)\in H_w^1$ it follows from our calculations that
  any G{\aa}rding vector $\pi(g)u$ is
  in $B_w$. Thus it is natural to use G{\aa}rding vectors
  in the discretization machinery of Feichtinger and Gr\"ochenig.
\end{remark}

\bibliographystyle{plain}
\bibliography{srkbslie}
\end{document}